\documentclass{amsart}
\usepackage{amssymb,latexsym}
\usepackage{amsmath}
\usepackage[dvipdfm]{graphicx}
 \usepackage{color}
\usepackage{enumerate}
\newenvironment{enumeratei}{\begin{enumerate}[\upshape (i)]}{\end{enumerate}}
\numberwithin{equation}{section}
\theoremstyle{plain}
 \newtheorem{theorem}{Theorem}[section]
 \newtheorem{lemma}[theorem]{Lemma}
 \newtheorem{proposition}[theorem]{Proposition}
 \newtheorem{corollary}[theorem]{Corollary}
\theoremstyle{definition}
 
 \newtheorem{remark}[theorem]{Remark}
 \newtheorem{claim}[theorem]{Claim}
 \newtheorem{example}[theorem]{Example}
 
\newcommand \url [1] {\tt{#1}}
\newcommand \geom [1] {\textup{Geom}(#1)} 
\newcommand \geomd [1]{\pair {\Jir{#1}}{\set{\Jir L\cap \ideal x: x\in L}}} 
\newcommand\ideal[1]{\mathord\downarrow #1}
\newcommand\filter[1]{\mathord\uparrow #1}
\newcommand \rhullo [1] {\textup{ConvH}_{\mathbb R^{#1}}} 
\newcommand \phullo [2]  {\textup{Hull}^{(#1)}_{#2}}  
\newcommand \phulla [3] {\phullo{#1}{#2}(#3)}  
\newcommand \chullo [1]  {\textup{Hull}^{\textup o}_{#1}}
\newcommand \chulla [2] {\chullo{#1}(#2)} 
\newcommand \powset [1] {\textup{PowSet}(#1)}
\newcommand \tuple [1] {\langle #1\rangle}
\newcommand \pair [2] {\tuple{#1,#2}}
\renewcommand\phi{\varphi}
\renewcommand\emptyset{\varnothing}
\newcommand \tbf [1] {\textbf{#1}} 
\newcommand \lat [1] {\textup{Lat}#1}    
\newcommand \set[1] {\{#1\}}
\newcommand \bigset[1] {\bigl\{#1\bigr\}}
\newcommand \rad [1] {\textup{rad}(#1)}
\newcommand \lmpt [1] {\textup{LPt}(#1)}
\newcommand \rmpt [1] {\textup{RPt}(#1)}
\newcommand \lend [1] {\textup{LEnd}(#1)}
\newcommand \rend [1] {\textup{REnd}\bigl(#1\bigr)}
\newcommand \points [1] {{^{\textup{ps}}\kern-2.5pt{#1}}}
\newcommand \pointss [1] {{^{\textup{ps}}\kern-1pt{#1}}}
\newcommand\fromto [3] { \textup{HInt}_{#1}(#2,#3)} 
\newcommand \lowstar [1] {{#1}_\ast}
\newcommand \width [1] {\textup{width}(#1)}
\newcommand \cprop [1]{$(\textup{CC}_{#1})$}
\newcommand \ephi {\widehat\phi} 
\newcommand \epsi {\widehat\psi} 
\newcommand \Jir [1] {\textup{Ji}\,#1} 
\newcommand \Mir [1] {\textup{Mi}\,#1}

\renewcommand\phi{\varphi}
\newcommand \Jleft {J_{\textup{left}}} 
\newcommand \Jbelow {J_{\textup{below}}}
\newcommand \then {\mathrel{\Rightarrow}} 
\newcommand \scsc[1] {\scriptscriptstyle{#1}}
\newcommand \izom [1]{\mathbf I\,#1}
\newcommand \kccc { \mathcal K_{\scsc{\textup{collinear}}}^{\scsc{\textup{concave}}} }
\newcommand \kccol { \mathcal K_{\scsc{\textup{collinear}}} }
\newcommand \kcplan { \mathcal K_{\scsc{\textup{planar}}} }
\newcommand \prel[1] { \mathcal R_{\scsc{\textup{points}}}^{\scsc{\textup{dim=}#1}}}
\newcommand \kcall { \mathcal G_{\scsc{\textup{all}}} }
\newcommand \nonparallel {\mathrel{
\not\mathord{\kern -1.5 pt\parallel}}}

\newcommand\init [1] {} 
\newcommand\nothing [1] {}
\newcommand\red[1]{{\textcolor{red}{#1}}}
\newcommand\blue[1]{{\textcolor{blue}{#1}}}
\newcommand\green[1]{{\textcolor{green}{#1}}}
\newcommand\skc[1]{{\textcolor{magenta}{#1}}}
\newcommand\refket[1] {\red{#1}}
\newcommand\refegy[1] {\blue{#1}}
\newcommand\refboth[1]{\green{1}}

\begin{document}
\title[Finite convex geometries of circles]
{Finite convex geometries of circles}

\author[G.\ Cz\'edli]{G\'abor Cz\'edli}
\email{czedli@math.u-szeged.hu}
\urladdr{http://www.math.u-szeged.hu/$\sim$czedli/}
\address{University of Szeged, Bolyai Institute. 
Szeged, Aradi v\'ertan\'uk tere 1, HUNGARY 6720}

\thanks{This research was supported by the NFSR of Hungary (OTKA), grant numbers  K77432 and
K83219}


\subjclass[2010]{Primary 05E99;   secondary 06C10} 
\nothing{
05E99 Algebraic combinatorics (1991-now) None of the above, but in this section;
06C10 (1980-now) Semimodular lattices, geometric lattices}

\keywords{convex geometry, anti-exchange property, geometry of circles, lower semimodular lattice, planar lattice}

\date{December 14, 2012; revised May 20, 2013}

\begin{abstract} 
Let $F$ be a finite set of  circles in the plane. We point out that the usual convex closure restricted to $F$ yields a convex geometry, that is, a combinatorial structure introduced by P.\,H.\ Edelman in 1980 under the name ``anti-exchange closure system''. We prove that if the circles are collinear and they are  arranged in a ``concave  way'', then they determine a convex geometry of convex dimension at most 2, and each finite convex geometry of convex dimension at most 2 can be represented this way. 
The proof uses some recent results from Lattice Theory, and some of the auxiliary statements on lattices or convex geometries could be of separate interest. The paper is concluded with some open problems.
\end{abstract}

\maketitle

\section{Introduction}\label{introsection}
\subsection{Aim and motivation}
\marginpar{All changes after December 14, 2012, are  in 
\skc{magenta}, \refegy{blue}, or \refket{red}.} 
The concept of convex geometries was introduced by \init{P.\,H.\ }Edelman~\cite{edelman} and \cite{edelmanproc}, see also \init{P.\,H.\ }Edelman and \init{R.\,E. }Jamison~\cite{edeljam},  \init{K.\ }Adaricheva, \init{V.\,A.\ }Gorbunov, and \init{V.\,I.\ }Tumanov \cite{r:adarichevaetal}, and  \init{D.\ }Armstrong \cite{armstrong}. \emph{Convex geometries} are combinatorial structures: finite sets with anti-exchange closures such that the emptyset is closed. 
They are equivalent to antimatroids, which are particular greedoids, and also to meet-distributive lattices. 
Actually, the concept of convex geometries has many equivalent variants. The first of these variants is due to \init{R.\,P.\ }Dilworth \cite{r:dilworth40}, and the early ones were surveyed in \init{B.\ }Monjardet \cite{monjardet}.  Since it would wander to far if we overviewed the rest, more than a dozen approaches, we only mention \init{K.\ }Adaricheva \cite{adaricheva},
\init{H.\ }Abels \cite{abels}, \init{N.\ }Caspard and \init{B.\ }Monjardet \cite{caspardmonjardet}, \init{S.\,P.\ }Avann \cite{avann}, 
\init{R.\,E.\ }Jamison-Waldner  \cite{jamisonwaldner}, and \init{M.\ }Ward \cite{ward}  for additional sources, and \init{M.\ }Stern\ \cite{stern}, 
\init{K.\ }Adaricheva and \init{G.\ }Cz\'edli \cite{adarichevaczg}, and 
\init{G.\ }Cz\'edli \cite{czgcoord} for some recent overviews. However, we need only a small part of the theory of convex geometries, and the present paper is intended to be self-contained for 
those who know the rudiments of Lattice Theory  up to, say, the concept of semimodularity.

From combinatorial point of view, the finite convex geometries are the interesting ones. Hence, and also because the tools we  use are elaborated only for the finite case, the present paper is restricted to \emph{finite} convex geometries.
Postponing the exact definition of convex geometries to Section~\ref{secketto}, we present an important finite  example as follows. Let $n\in \mathbb N=\set{1,2,3,\dots}$, and let $E$ be a finite subset of the $n$-dimensional space $\mathbb R^n$. The set of all subsets of $E$ is denoted by $\powset E$. For $Y\subseteq E$, we define $\phulla nEY=E\cap\rhullo n(Y)$, where $\rhullo n(Y)$ denotes the usual convex hull of $Y$ in $\mathbb R^n$. \skc{The} map $\phullo nE\colon \powset E\to\powset E$ is a closure operator, and $\pair E{\phullo nE}$ is a finite convex geometry. 
Convex geometries of this form are called \emph{geometries of relatively convex sets}, and they (not only the finite ones) were  studied by 
\init{K.\ }Adaricheva \cite{adarichevaproc} and \cite{adaricheva},  \init{G.\,M.\ }Bergman \cite{bergman}, and \init{A.\,P.\ }Huhn \cite{huhn}. 

We know from \init{G.\,M.\ }Bergman \cite{bergman} that each finite convex geometry $G$ can be \skc{\emph{embedded}} into a geometry $\pair E{\phullo nE}$ of relative convex sets; $n$ depends on $G$. 
However, even if a finite convex geometry is of convex dimension 2, it is not necessarily isomorphic to some $\pair E{\phullo nE}$.

\skc{Besides geometries of relatively convex sets, there exists a more complicated way to define a convex geometry on a subset $B\subseteq \mathbb R^n$ by means of the 
 usual convex hull operator $\rhullo n$.
For definition, let $B$ and $A$ be finite subsets, acting as a base set and an auxiliary set,  of $\mathbb R^n$ such that $\rhullo n(A)\cap B=\emptyset$. For  $X\subseteq B$, let $\phullo {n,A}B(X)=B\cap\rhullo n(X\cup A)$.
By \refket{Kenji Kashiwabara, Masataka Nakamura and Yoshio Okamoto~\cite{nakamuraatall}}, 
$\pair B{\phullo {n,A}B}$ is a convex geometry and, moreover,  each finite convex geometry is \emph{isomorphic} to an appropriate $\pair B{\phullo {n,A}B}$.}

\refegy{Motivated by the results of \cite{bergman} and \cite{nakamuraatall} mentioned above, the present paper  introduces another kind of ``concrete'' finite convex geometries  that are still  based on the usual concept $\rhullo n$ of convexity.
However, our primary purpose is to represent finite convex geometries in a \emph{visual, conceptually simple} way. In particular, we look for a representation theorem that leads to readable figures, at least in case of small size, because figures are generally useful in  understanding a subject. (This is well exemplified by the role that Hasse diagrams play, even if the present paper cannot compete with their importance.)  The space $\mathbb R^n$ for $n\geq 3$ can hardly offer comprehendible figures. The real line $\mathbb R^1=\mathbb R$ is too ``narrow'' to hold overlapping objects in a readable way, and only few convex geometries can be represented by it. Therefore, with the exception of Subsection~\ref{kirasection}, we will only work in the plane $\mathbb R^2$. The plane is  general enough to represent all finite convex geometries  of convex dimension 2, and also some additional ones.} 

\skc{To accomplish our goal,} we start from a finite set $F$ of circles in the plane $\mathbb R^2$, and we define a convex geometry $\pair F{\chullo F}$ with the help of forming usual convex hulls in the plane, analogously to the geometries of relative convex sets. The structures $\pair F{\chullo F}$, called \emph{convex geometries of circles}, are very close to the usual closure $\rhullo 2\colon \mathbb R^2\to\mathbb R^2$ and, as opposed to geometries of relatively convex sets, we can prove  that each finite  convex geometry of convex dimension at most 2 is isomorphic to some convex geometry $\pair F{\chullo F}$ of circles. Actually, our representation theorem, the main result of the paper, will assert more by imposing some conditions on $F$; see Figure~\ref{fig1} (without $C_3'$, $D$, the grey-colored plane shape $H$, and the dotted curves) for a first impression.

\begin{figure}
\centerline
{\includegraphics[scale=1.0]{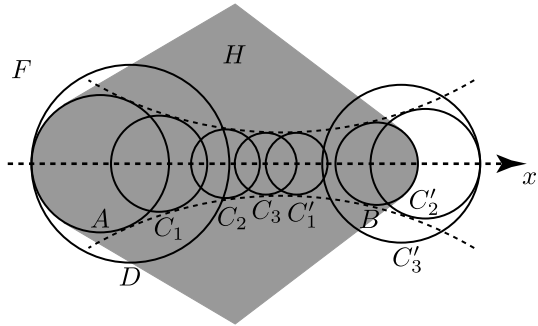}}
\caption{A concave set of collinear circles (disregard the grey-colored shape $H$ until the proof of Lemma~\ref{fromtolemma})
\label{fig1}}
\end{figure}

This paper uses some lattices as auxiliary tools in proving the main result. In fact, the  theory of slim semimodular lattices has rapidly developed recently, as witnessed by 
\init{G.\ }Cz\'edli \cite{czgmatrix}, \cite{czgrepres}, and \cite{czgasympt}, 
\init{G.\ }Cz\'edli, \init{T.\ }D\'ek\'any,  \init{L.\ }Ozsv\'art, \init{N.\ }Szak\'acs, and \init{B.\ }Udvari \cite{czgdekanyatall}, 
\init{G.\ }Cz\'edli and \init{G.\ }Gr\"atzer \cite{czgggresect}, 
\init{G.\ }Cz\'edli, {L.\ }Ozsv\'art, and  \init{B.\ }Udvari \cite{czgolub},
\init{G.\ }Cz\'edli and \init{E.\,T.\ }Schmidt \cite{czgschtJH}, \cite{czgschvisual}, \cite{czgschperm}, and \cite{czgschslim2}, 
\init{G.\ }Gr\"atzer and \init{E.\ }Knapp \cite{gratzerknapp}, \cite{gratzerknapp3}, and \cite{gratzerknapp4}, and 
\init{E.\,T.\ }Schmidt \cite{schmidtrepres}.
Even if only a part of the results in these papers are eventually needed here, this progress  provides the background of the present work.  On the other hand, some auxiliary statements we prove here, namely some of the propositions, seem to be of some interest in the theory of slim semimodular lattices, while some other propositions could be interesting in the  theory of convex geometries.

\subsection{Outline}Section~\ref{secketto} gives the basic concepts and formulates two results. Proposition~\ref{prop1} asserts that, for every finite set $F$ of circles in the \refegy{plane},  $\pair F{\chullo F}$  is a convex geometry, while our main result, Theorem~\ref{thmmain}, is the converse statement for the case of convex dimension $\leq 2$. The results of Section~\ref{secketto} are proved in Section~\ref{proofsection}, where several auxiliary statements are cited or proved, and some more concepts are recalled. 
Section~\ref{oddsendch} contains examples, statements, and open problems to indicate that although  Theorem~\ref{thmmain} gives a satisfactory representation of  finite convex geometries of convex dimension at most 2, we are far \skc{both} from settling the case of higher convex dimensions \skc{and from understanding what the abstract class of our convex geometries of circles is}. 
%

\subsection{Prerequisites} As mentioned already, the reader is only assumed a little knowledge of lattices. Besides the first few pages of any book on lattices or particular lattices, including \init{G.\ }Gr\"atzer \cite{GGLT}, \init{J.\,B.\ }Nation \cite{nationbook}, and \init{M.\ }Stern \cite{stern}, even the first chapter of 
\init{S.\ }Burris and \init{H.\,P.\ }Sankappanavar \cite{burrissankapp}, which does not even focus on Lattice Theory, is sufficient. Note that \cite{burrissankapp} and \cite{nationbook} are freely downloadable.

\section{Convex geometries and our results}\label{secketto}
\subsection{Basic concepts and the first result}
Assume that we are given a set $U$ and a map $\Phi\colon \powset U\to \powset U$. 
If $X\subseteq \Phi(X)=\Phi(\Phi(X))\subseteq \Phi(Y)$ holds for all $X\subseteq Y\subseteq U$, then $\Phi$ is a \emph{closure operator} on $U$. If $\Phi$ is a closure operator on $U$, $\Phi(\emptyset)=\emptyset$, and $\Phi$ satisfies the so-called \emph{anti-exchange property}
\begin{equation}\label{aeprogor}
\begin{aligned}
&\text{if $\Phi(X)=X\in \powset U$, $x,y\in U\setminus X$, $x\neq y$,}
\cr
&\text{and $x\in \Phi(X\cup\set y)$, then $y\notin \Phi(X\cup\set x)$,}
\end{aligned}
\end{equation}
then $\pair U\Phi$ is a \emph{convex geometry}. 
Given a convex geometry $\pair U\Phi$, we use the notation  
\[\lat{\pair U\Phi}=\set{X\in\powset U: X=\Phi(X)}
\] 
to denote the \emph{set of closed sets} of $\pair U\Phi$. Actually, the structure $\lat{\pair U\Phi}=\bigl\langle \lat{\pair U\Phi}, \subseteq \bigr\rangle$ is a \emph{lattice}, and it is a complete meet-subsemilattice of the powerset lattice $\powset U=\pair{\powset U}{\cup,\cap}$.  It is well-known that $\lat{\pair U\Phi}$ determines $\pair U\Phi$ since we have $\Phi(X)=\bigcap\set{Y\in \lat{\pair U\Phi}: X\subseteq Y}$. Hence, it is natural to say that $\pair U\Phi$ can be \emph{embedded} into or \emph{isomorphic} to a convex geometry $\pair V\Psi$, if the lattice  $\lat{\pair U\Phi}$  can be embedded into or isomorphic \refket{to} $\lat{\pair V\Psi}$, respectively. 

As usual,  a \emph{circle} is a set $\set{\pair xy\in \mathbb R^2 : (x-u)^2+(y-v)^2=r }$, where $u,v,r\in\mathbb R$ and $r\geq 0$. A circle of radius 0 consists of a single point. Since $\rhullo 2$, the operator of forming convex hulls, is defined for subsets of $\mathbb R^2$ rather than for sets of circles, we introduce a shorthand notation for ``\emph point\emph s of'' (or ``\emph point \emph set of'') as follows. 
For a set $X$ of circles in $\mathbb R^2$, the \emph{set of points} belonging to some member of $X$ is denoted by $\points X$. In other words, 
\begin{equation}
\points X=\bigcup_{C\in X}C\text.\label{psxdFa}
\end{equation} 
\skc{For a set $F$ of circles in $\mathbb R^2$ and $X\subseteq F$,} we define 
\begin{equation}
\chulla FX= \bigset{C\in F: C\subseteq \rhullo 2{\bigl( \points X \bigr)  }  }\text.\label{psxdFb}
\end{equation}
(The superscript circle in the notation will remind us that $\chullo F$ is defined on a set of circles.)
The structure $\pair F{\chullo F}$ will be called the \emph{convex geometry} of $F$, and we call it  a \emph{convex geometry of circles} if $F$ is not specified. (We shall soon prove that it is a convex geometry.)  Note that if all circles of $F$ are of radius 0 and $E\subseteq \mathbb R^2$ is the set of their centers, then  $\pair F{\chullo F}$   is obviously isomorphic to $\pair E{\phullo 2E}$. 

We are now in the position to state our first observation; the  statements of this section will be proved in Section~\ref{proofsection}.

\begin{proposition}\label{prop1}
For every finite set $F$ of circles in $\mathbb R^2$,   $\pair F{\chullo F}$  is a convex geometry.
\end{proposition}

\subsection{Collinear circles and the main result}\label{colsmss}
If there is a line containing the centers of all members of $F$ above, then $F$ is a \emph{set of collinear circles}. 
For simplicity, we will always assume that the line in question is the $x$ axis. That is, in case of a set of collinear circles, all the centers are of the form $\pair u0$. 
For example, $F$ in Figure~\ref{fig1}
is a set of collinear circles; note that the dotted curves and the $x$-axis do not belong to $F$.  \refegy{
Note that a set of collinear circles can always be given by a set of intervals of the real line $\mathbb R$; this comment will be expanded in Subsection~\ref{kirasection}. However, circles lead to a stronger result and more readable figures than intervals.}
The label of a circle in our  figures is either below the center  (inside or outside but close to the circle), or we use an arrow. 
The \emph{radius} of a circle $C$ is denoted by $\rad C$. If the center  of $C$ is  $\pair u0$, then 
\[ \lmpt C=u-\rad C\,\,\text{ and }\,\,
\rmpt C=u+\rad C
\]
will denote the \emph{leftmost point} and the \emph{rightmost point} of $C$, respectively. Since we allow that two distinct circles have the same leftmost point or the same rightmost point, we also need the following concept. 
Although circles are usually treated as endless figures, in case the  center of a circle $C$ lies on the $x$ axis,   we define the \emph{left end} and the \emph{right end} of $C$ as follows:
\[\lend C=\pair{\lmpt C}{-\rad C}
\,\,\text{ and }\,\, \rend C=\pair{\rmpt C}{\rad C}\text.
\]
Left and right ends are ordered lexicographically; this order is denoted by $\sqsubset$. Thus
{\allowdisplaybreaks{
\begin{align*}
\lend{C}\sqsubset\lend{D}\kern4pt \iff \kern4pt &\lmpt C<\lmpt D, \text{ or}\cr
 &\lmpt C=\lmpt D\text{ and } 
 \rad C>\rad D,\cr
\rend{C}\sqsubset\rend{D}\kern4pt \iff \kern4pt &\rmpt C<\rmpt D, \text{ or}\cr
 &\rmpt C=\rmpt D\text{ and } 
 \rad C<\rad D\text.
\end{align*}
}}%
For later reference, note the obvious rules:
\begin{equation}\label{obvruuules}
\begin{aligned}
 \lmpt C<\lmpt D&\then \lend C\sqsubset\lend D\text{ and }\cr
\rmpt C<\rmpt D&\then \rend C\sqsubset\rend D\text.
\end{aligned}
\end{equation}
Let $F$ be a set  of collinear circles. 
We say that $F$ is a \emph{concave set of collinear circles} if for all $C_1,C_2,C_3\in F$,
\begin{equation}\label{efbetween}
\begin{aligned}
&\text{whenever 
$\lend{C_1} \sqsubset \lend{C_2}$ and $\rend{C_2} \sqsubset \rend{C_3}$,}\cr
&\text{then $C_2\subseteq \rhullo2\bigl(  {C_1}\cup {C_3} \bigr)$\text.
}
\end{aligned}
\end{equation}
For illustration, 
see $C_1$, $C_2$, and $C_3$, or $C_1'$, $C_2'$, and $C_3'$, in  Figure~\ref{fig1}. Note that since
each $C\in F$ is uniquely determined by $\lend C$ and also by $\rend C$, 
``$\sqsubset$'' in \eqref{efbetween} can  be replaced by ``$\sqsubseteq$''.
If $|\set{\lmpt{C_1},\lmpt{C_2} ,\rmpt{C_1},\rmpt{C_2}}|=4$
for any two distinct $C_1$ and $C_2$ in a set $F$ of collinear circles, then $F$ is called \emph{separated}. For example, $F\setminus\set{C_3', D}$ in Figure~\ref{fig1} is a separated, concave set of collinear circles.
Clearly, if $F$ is a separated set of collinear circles, then $F$ is concave if{f}
\begin{equation}\label{efptetween}
\begin{aligned}
&\text{whenever 
$\lmpt{C_1} < \lmpt{C_2}$ and $\rmpt{C_2} < \rmpt{C_3}$,}\cr
&\text{then $C_2\subseteq \rhullo2\bigl(  {C_1}\cup {C_3} \bigr)$.
}
\end{aligned}
\end{equation}

For a finite lattice $L$, the set of elements with exactly one lower cover, that is the set of non-zero join-irreducible elements, is denoted by $\Jir L$. Dually, $\Mir L$ stands for the set of elements with exactly one cover. The \emph{convex dimension} of a finite convex geometry $\pair U\Phi$ is the least integer $n$ such that  $\Mir{\!\bigl(\lat{\pair U\Phi}\bigr)}$ is the union of $n$ chains.  In other words, the convex dimension is the \emph{width} of the poset $\Mir{\!\bigl(\lat{\pair U\Phi}\bigr)}$. 

We are now in the position of formulating our main result, which characterizes finite convex geometries of convex dimension at most 2 
as the convex geometries of fini\skc{te} (separated or not necessarily separated) concave 
sets of collinear circles. 
\goodbreak

\begin{theorem}\label{thmmain}\  
\begin{itemize}
\item[{\textup(A)}]
If $F$ is a finite, concave set of  collinear circles in the plane, then $\pair F{\chullo F}$ is a convex geometry of convex dimension at most 2.
\item[{\textup(B)}] For each finite convex geometry $\pair U\Phi$ of convex dimension at most 2, there exists a finite,   separated, concave set $F$ of  collinear circles in the plane such that $\pair U\Phi$ is isomorphic to $\pair F{\chullo F}$.
\end{itemize}
\end{theorem}

\section{Proofs and auxiliary statements}\label{proofsection}
\subsection{Not necessarily collinear circles}
\begin{proof}[Proof of Proposition~\ref{prop1}] 
For every $X\subseteq F$, we have
\begin{equation}\label{dblrzlsr}
\text{$\rhullo 2(\points X)=\rhullo 2{\bigl(\pointss{( \chulla FX)} \bigr)}$}
\end{equation}
since the ``$\subseteq$'' inclusion follows from $X\subseteq \chulla FX$  while the converse inclusion comes from the obvious $ \rhullo 2(\points X) \supseteq\pointss(\chulla FX)$.

\begin{figure}
\centerline
{\includegraphics[scale=1.0]{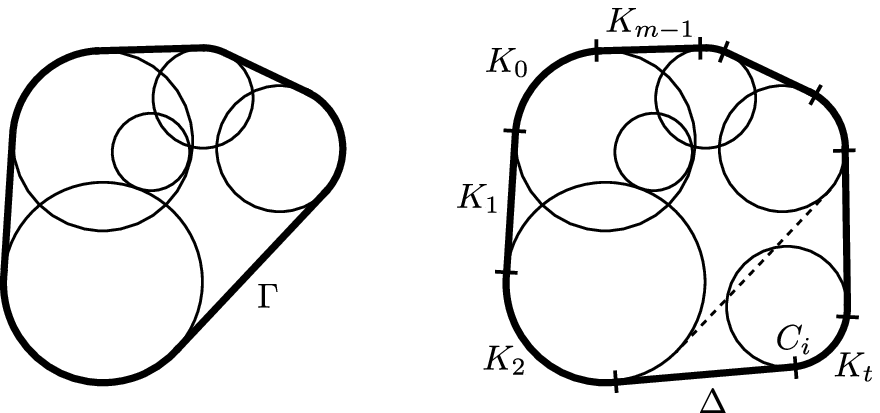}}
\caption{Illustrating the proof of Proposition~\ref{prop1}\label{fig2}}
\end{figure}

Assume that $X\subseteq F$, and let  $Y=\chulla FX$. Since $C\subseteq \rhullo 2(\points X)$ holds for all $C\in Y$, that is $\pointss Y\subseteq  \rhullo 2{( \points  X )}$, we obtain  
\[\rhullo 2{( \pointss  Y )} \subseteq   \rhullo 2{( \points  X  )}\text.\]
We have equality here since $Y\supseteq X$. This implies $\chulla FY=\chulla FX$, and it follows that $\chullo F$ is a closure operator with the property $\chulla F\emptyset=\emptyset$.
Observe that a closure operator $\Phi$ on $U$ satisfies   the anti-exchange property \eqref{aeprogor} if{f}
\begin{equation}\label{aepropnw}
\begin{aligned}
&\text{if $\Phi(X)=X\in \powset U$, $x_0,x_1\in U\setminus X$,}
\cr
&\text{and  $\Phi(X\cup\set{ x_0})=\Phi(X\cup\set {x_1})$, then $x_0=x_1$.}
\end{aligned}
\end{equation}
Hence, tailoring \eqref{aepropnw} to our situation, we assum\skc{e}   $X=\chulla FX \in \powset F$, $\set{C_0,C_1}\subseteq F\setminus X$, and $\chulla F{X\cup\set{C_{0}}}=\chulla F{X\cup\set{C_{1}}}$. 
We have to show $C_0=C_1$.
Combining our assumption with \eqref{dblrzlsr}, we obtain 

\begin{equation}\label{osdkHr}
\rhullo2 \bigl({ C_0 \cup\points X }\bigr) = 
\rhullo2 \bigl({ C_1 \cup\points X }\bigr)\text.
\end{equation} 
Let $\Gamma$ and $\Delta$ be the boundary of $\rhullo2 (\points X)$ and that of the set given in \eqref{osdkHr}, respectively; see the thick closed curves in Figure~\ref{fig2}. In the figure, $X$ is depicted twice: in itself on the left and with $C_i$ on the right. We can imagine $\Gamma$ \skc{and $\Delta$} as \skc{tight} resilient rubber noose\skc{s} around the members of $X$ \skc{and $X\cup\set{C_i}$, respectively}. Pick an $i\in\set{0,1}$.
Observe that  $C_i\notin X=\chulla FX$ implies $C_i\not\subseteq \rhullo2 (\points X )=\skc{\rhullo2 (\Gamma)}$. \skc{Clearly, $\Delta$ can be decomposed into finitely many segments $K_0,K_1,\dots, K_{m-1}$, listed anti-clockwise, such that these segments are of positive length and the following properties hold for all $j\in\set{0,\dots, m-1}$ and $i\in \set{0,1}$: 
\begin{enumeratei}
\item The endpoint of $K_j$ is the first point of $K_{j+1}$, where $j+1$ is understood modulo $m$.
\item  Either $K_j\subseteq \Gamma$, or no inner point of $K_j$ belongs to $\rhullo2 (\Gamma)$.
\item\label{dwzzr} If no inner point of $K_j$ belongs to $\rhullo2 (\Gamma)$, then either $K_j$ is a straight line segment, or $K_j$ is an arc of $C_i$.
\item\label{dwrkdt} There exists a $t\in \set{0,\dots,m-1}$ such that either
\begin{enumeratei}
\item[(a)] $K_t$ is not a straight line segment  and none of its inner points belongs to $\rhullo2 (\Gamma)$; or
\item[(b)]  both $K_t$ and $K_{t+1}$ are straight line segments and their common point is outside $\rhullo2 (\Gamma)$.
\end{enumeratei}
\end{enumeratei}
As opposed to Figure~\ref{fig2}, note that $t$ in \eqref{dwrkdt} need not be unique.
Clearly,  \eqref{dwrkdt}(b) holds only if{f} the radius of $C_i$ is zero; in this case $C_i$ consists of the common point of $K_t$ and $K_{t+1}$, which  is the only angle  point of $\Delta$ outside $\rhullo2 (\Gamma)$, and we conclude $C_0=C_1$. Otherwise, if \eqref{dwrkdt}(a) holds, then \eqref{dwzzr} implies that $K_t$ is a common arc of $C_0$ and $C_1$, whence we conclude $C_0=C_1$ again. Therefore, $\chullo F$ is an anti-exchange closure operator, and  $\pair F{\chullo F}$ is a convex geometry.} 
\end{proof}

\subsection{Collinear circles}
For a  set $F$ of collinear circles and  $A,B\in F$, we define the \emph{horizontal interval}
\begin{align*}\fromto FAB =  \set{ C\in F: \lend A\sqsubseteq \lend C\text{,  }\rend C \sqsubseteq \rend B} \text.
\end{align*} 
Note that $\fromto FAB$ can be empty. If $A\neq B$ and $A$ is inside the circle $B$, then $A\in \fromto FAB\cap \fromto FBA$   but $B\notin \fromto FAB\cup\fromto FBA$. Note also that, for  $A\in F$, the horizontal interval $\fromto FAA$ is the \emph{set of circles} of $F$ that are \emph{inside}  $A$, including $A$ itself.

\begin{lemma}\label{fromtolemma}
If $F$ is a finite set of  collinear circles in the plane, then the following three statements hold.
\begin{enumeratei}
\item\label{fromtolemmaa}  $\refket{\set\emptyset}\cup \set{\fromto FAB: A,B\in F}\subseteq \lat{\pair F{\chullo F}}$. If $F$ is concave, then even the equality  $\refket{\set\emptyset} \cup \set{\fromto FAB: A,B\in F}= \lat{\pair F{\chullo F}}$ holds.
\item\label{fromtolemmab} For $\emptyset\neq X\in  \lat{\pair F{\chullo F}}$,  let $A$ and $B$ denote the circles in $X$ with least left end and greatest right end, respectively\skc{. If} $F$ is concave\skc{, t}hen \skc{the equality}
$X=\fromto FAB = \fromto FAA\vee\fromto FBB
$ holds in $\lat{\pair F{\chullo F}}$.
\item\label{fromtolemmac} If $X\in \lat{\pair F{\chullo F}}$,  then $X=\bigvee\set{ \fromto FCC: C\in X}$.
\end{enumeratei}
\end{lemma}

\begin{proof} We can assume $F\neq \emptyset$ since otherwise the lemma is trivial.
Assume $\fromto FAB\neq\emptyset$. 
Clearly, as indicated in Figure~\ref{fig1} by the grey-colored plane shape, there exists a convex subset $H$ of $\mathbb R^2$ such that $\pointss{\bigl(\fromto FAB\bigr)}\subseteq H$ 
but $H$ includes no $C\in F\setminus \fromto FAB$
as a subset. (As in the figure, it is always possible to find an appropriate $H$ whose boundary  is the union of an arc of $A$, that of $B$, and four straight line segments. Note that if $\rad A=0$, then the only arc of $A$ is itself.) 
This implies that  $\fromto FAB  = \chulla F{\fromto FAB}  \in  \lat{\pair F{\chullo F}}$. Thus  
$\refket{\set{\emptyset}}\cup \set{\fromto FAB: A,B\in F}\subseteq\lat{\pair F{\chullo F}}$. 

To prove the converse inclusion under the additional assumption that $F$ is concave, and let $\emptyset \neq X\in \lat{\pair F{\chullo F}}$. Since $F$ is finite, there are a unique $A\in X$ with least left end  $\lend A$ and a unique $B\in X$ with largest right end $\rend B$.
Using that $F$ is concave, we obtain that, for every $C\in \fromto FAB$,  
$C\subseteq \rhullo 2(A\cup B) \subseteq \rhullo 2({\points{X} }) $, that is, $C\in\chulla F{X}$. This implies 
$\fromto FAB\subseteq  \chulla FX=X$. On the other hand, if $C\in X$, then 
the choice of $A$ and $B$ gives 
$\lend A\sqsubseteq  \lend C$ and $\rend C\sqsubseteq  \rend B$, that is, $C\in \fromto FAB$.
Hence, $X=\fromto FAB$, and we obtain 
$\lat{\pair F{\chullo F}}\subseteq \refket{\set{\emptyset}}\cup \set{\fromto FAB: A,B\in F}$. This proves  \eqref{fromtolemmaa} and the first equality in \eqref{fromtolemmab}. 

Next, we do not assume that $F$ is concave. If $C\in X\in \lat{\pair F{\chullo F}}$, then $D\subseteq \rhullo 2(C)\subseteq \rhullo2(\pointss X)$ holds for all $D\in\fromto FCC$. That is, $D\in\chulla FX=X$ for all $D\in\fromto FCC$. 
Thus, for all $C\in X$, we have  $\set C\subseteq \fromto FCC\subseteq X$, which clearly implies \eqref{fromtolemmac}. 

Finally, assume again that $F$ is concave. From \eqref{fromtolemmac} and the first equality in \eqref{fromtolemmab}, we obtain  $X=\fromto FAB\supseteq \fromto FAA\vee\fromto FBB$. \refket{The} reverse inclusion follows from the assumption that $F$ is concave. Hence, \eqref{fromtolemmab} holds.
\end{proof}

\begin{lemma}\label{mzrrDcgbn}
If $F$ is a finite  set of collinear  circles in the plane, then \[\Jir{\!\bigl(\lat{\pair F{\chullo F}}\bigr)} =\set{\fromto FAA: A\in F}\text.\]
\end{lemma}

\begin{proof} Although the statement is intuitively more or less clear, we give an exact proof. 
First, we show that, for $A\in F$, $\fromto FAA\in \Jir{\!\bigl(\lat{\pair F{\chullo F}}\bigr)}$. This is obvious if $|\fromto FAA|=1$, that is, if 
$\fromto FAA$ is an atom in $\lat{\pair F{\chullo F}}$. Assume $|\fromto FAA|\geq 2$. Let $B$ and $C$ be the circles with least left end and  greatest right end in $\fromto FAA\setminus\set A$, respectively. (They need not be distinct.)  Since the circles of $F$ are determined by their left ends, $\lend A\neq \lend B$, and we obtain $\lend A\sqsubset \lend B$ by the definition of $\fromto FAA$. Similarly, $\rend C\sqsubset \rend A$. Now if $D\in\fromto FAA\setminus A$, then the choice of $B$ and $C$ gives $D\in\fromto FBC$. 
On the other hand, if $D\in\fromto FBC$, then  
\[\lend A\sqsubset \lend B\sqsubseteq \lend D\text{ and } \rend D\sqsubseteq \rend C\sqsubset \rend A\]
yield $D\in\fromto FAA\setminus A$. Thus 
$\fromto FAA\setminus A= \fromto FBC\in \lat{\pair F{\chullo F}}$ is a unique lower cover of $\fromto FAA$, because $A\in \skc{X}$ implies $\fromto FAA\subseteq X$ for all $X\in \skc{\lat{\pair F{\chullo F}}}$. Hence,  $\fromto FAA\in \Jir{\!\bigl(\lat{\pair F{\chullo F}}\bigr)}$. 
This proves
$\Jir{\!\bigl(\lat{\pair F{\chullo F}}\bigr)} \supseteq \set{\fromto FAA: A\in F}$.

To prove the converse inclusion, let $X\in\Jir{\bigl(\lat{\pair F{\chullo F}}\bigr)}$.
We obtain from Lemma~\ref{fromtolemma}\eqref{fromtolemmac} that $X=\fromto F{\skc A}{\skc A}$ for some ${\skc A}\in F$.  Thus $\Jir{\bigl(\lat{\pair F{\chullo F}}\bigr)}\subseteq \set{\fromto F{\skc A}{\skc A}: {\skc A}\in \skc F}$.
\end{proof}

\begin{lemma}\label{mirlemma}
Assume that $F$ is a finite,  concave set of collinear circles in the plane, and $|F|\geq \skc{1}$. \skc{If} the circle in $F$ with least left end and that with largest right end \skc{are} denoted by $K_l$ and $K_r$, respectively\skc{, then} 
\[ \skc{\set{\emptyset}}\cup \set{\fromto F{K_l}B: B\in F}\text{ and }
\skc{\set{\emptyset}}\cup \set{\fromto FA{K_r}: A\in F}
\]
are chains in $\lat{\pair F{\chullo F}}$, and each nonempty $X\in \lat{\pair F{\chullo F}}$ is of the form
\[X=\fromto F{K_l}B \wedge \fromto FA{K_r}\skc{\text{, that is, }} \fromto F{K_l}B \cap \fromto FA{K_r},\] 
where $A,B\in X$ are defined in Lemma~\ref{fromtolemma}\eqref{fromtolemmab}.
\end{lemma}

\begin{proof} Let $B_1,B_2\in F$. We can assume $\rend {B_1}\sqsubseteq  \rend {B_2}$ since ``$\sqsubseteq$'' is a linear order. By transitivity, $\fromto F{K_l}{B_1} \subseteq  \fromto F{K_l}{B_2}$. 
Therefore,  the first set in the lemma is a chain. By left-right duality, so is the second one. 

Next, let $\emptyset\neq X\in \lat{\pair F{\chullo F}}$. \skc{We have} $X=\fromto FAB$ by Lemma~\ref{fromtolemma}\eqref{fromtolemmab}, and the obvious \skc{equality} $\fromto FAB = \fromto F{K_l}B \cap \fromto FA{K_r}$ completes the proof. 
\end{proof}

\subsection{Lattices associated with convex geometries} One of the many equivalent ways, actually the first way, of defining convex geometries was to use meet-distributive lattices; see \init{R.\,P.\ }Dilworth \cite{r:dilworth40}.
Now we recall some concepts from Lattice Theory.

A finite lattice $L$ is \emph{lower semimodular} if whenever $a,b\in L$ such that $a$ is covered by $a\vee b$, in notation $a\prec a\vee b$, then $a\wedge b\prec b$. Equivalently, if the implication $a\prec b\then a\wedge c\preceq b\wedge c$ holds for all $a,b,c\in L$.
We will often use the trivial fact that 
this property is inherited by intervals and, more generally, by cover-preserving sublattices. 
For $u\neq 0$ in a finite lattice $L$, let $\lowstar u$ denote the meet of all lower covers of $u$. A finite lattice $L$ is \emph{meet-distributive} if the interval $[\lowstar u,u]$ is a distributive lattice for all $u\in L\setminus \set0$.
For other definitions,  see  \init{K.\ }Adaricheva~\cite{adaricheva},  \init{K.\ }Adaricheva, \init{V.\,A.\ }Gorbunov and \init{V.\,I.\ }Tumanov~\cite{r:adarichevaetal}, and 
\init{N.\ }Caspard and \init{B.\ }Monjardet~\cite{caspardmonjardet}; see also  \init{G.\ }Cz\'edli~\cite[Proposition 2.1 and Remark 2.2]{czgcoord} and \init{K.\ }Adaricheva and \init{G.\ }Cz\'edli \cite{adarichevaczg} for recent surveys and developments. 

The study of  meet-distributive lattices (and their duals)  goes back to \init{R.\,P.\ }Dilworth \cite{r:dilworth40}, 1940. There were  a lot of discoveries and rediscoveries of these lattices and equivalent combinatorial structures (including convex geometries); see  \cite{r:adarichevaetal} and \cite{czgcoord}, \init{B.\ }Monjardet~\cite{monjardet}, and \init{M.~}Stern~\cite{stern} for surveys. We recall the following statement; its origin is  the combination of \init{M.\ }Ward~\cite{ward} (see also \init{R.\,P.~}Dilworth~\cite[page 771]{r:dilworth40}, where \cite{ward} is cited) and \init{S.\,P.~}Avann~\cite{avann} (see also 
\init{P.\,H.\ }Edelman \cite[Theorem 1.1(E,H)]{edelmanproc}, where \cite{avann} is recalled). 
\begin{claim}\label{mdstrthnlsm}
Every finite meet-distributive lattice is lower semimodular.
\end{claim}

The \emph{width} of a partially ordered set $P$, denoted by $\width{P}$, is the smallest $k$ such that $P$ is  the union of $k$ appropriate chains of $L$. Equivalently, see 
\init{R.\,P.\ }Dilworth \cite{dilworth50}, $\width P$ is the largest $k$ such that 
there is a $k$-element antichain in $P$.
For a  finite lattice $L$, we are interested in the width of $\Mir L$. Note that, clearly,  $\width{\Mir L}$ is the smallest $k$ such that the union of $k$ maximal chains of $L$ includes $\Mir L$.  
Following \init{G.\ }Gr\"atzer and \init{E.\ }Knapp \cite{gratzerknapp} and, in the present form,  \init{G.\ }Cz\'edli and \init{E.\,T.\ }Schmidt \cite{czgschtJH},
finite lattices $L$ with $\width{\Mir L}\leq 2$  are called \emph{dually slim}.
Finite lattices $L$ with $\width{\Jir L}\leq 2$ are, of course, called \emph{slim}.

If $L$ is a lattice and $x\in L$, then the \emph{principal ideal} $\set{y\in L: y\leq x}$ is denoted by $\ideal x$.
We have already mentioned that a finite convex geometry $G=\pair U\Phi$ determines a lattice, the lattice $\lat{(G)}=\lat{\pair U\Phi}$ of its closed sets.
Conversely, if $L$ is a 
finite meet-distributive lattice, then we can take the combinatorial structure $\geom L=\geomd L$. 
Part of the following lemma, which asserts that finite convex geometries and finite meet-distributive lattices are essentially the same, 
was proved by  \init{P.\,H.\ }Edelman~\cite[Theorem 3.3]{edelman}, see also \init{D.\ }Armstrong~\cite[Theorem 2.8]{armstrong}. The rest can be extracted from \init{K.\ }Adaricheva, \init{V.\,A.\ }Gorbunov, and \init{V.\,I.\ }Tumanov\cite[proof of Theorem 1.9]{r:adarichevaetal}; see also \init{G.\ }Cz\'edli \cite[Lemma 7.4]{czgcoord} for more  details.

\begin{lemma}\label{geomversuslat} If $L$ is a finite meet-distributive lattice and  $G=\pair U\Phi$ is a  finite convex geometry, then the following three statements hold.
\begin{enumeratei}
\item\label{latamatEQa}  $\lat{(G)}$ is a finite meet-distributive lattice.
\item\label{latamatEQb} $\geom L$ is a  finite convex geometry.
\item\label{latamatEQc} 
$\lat{(\geom L)}  \cong L$ and $\geom{\lat {(G)}} \cong G$.
\end{enumeratei}
\end{lemma}

\subsection{Dually slim, lower semimodular lattices}
Finite, slim, semimodular lattices are more or less understood. Therefore, so are their duals, the dually slim, lower semimodular, finite lattices. The following lemma  is practically known, but we will  explain how to extract it from the literature. The notation introduced before Lemma~\ref{geomversuslat} is still in effect.

\begin{lemma}\label{slimgeomlemma} \skc{If} $\Phi$ \skc{is} a closure operator on $U$\skc{, t}hen 
 $G=\pair U\Phi$ is a  finite convex geometry of convex dimension at most 2 if{f}$\,$ $\lat{(G)}$ is a finite, dually slim, lower semimodular lattice.
\end{lemma}

\begin{proof} In view of Lemma~\ref{geomversuslat}, all we have to show is that, for a finite lattice $L$, the following two conditions are equivalent:
\begin{enumeratei}
\item\label{sdszdFa} $L$ is meet-distributive and $\width{\Mir L}\leq 2$;
\item\label{sdszdFb} $L$ is lower semimodular and dually slim.
\end{enumeratei}

With reference to \init{M.\ }Ward~\cite{ward} and \init{S.\,P.~}Avann~\cite{avann}, we have already mentioned that meet-distributivity implies lower semimodularity. Thus \eqref{sdszdFa} implies \eqref{sdszdFb}. Conversely, 
assume \eqref{sdszdFb}. \skc{We conclude that} $L$ is meet-distributive by the 
dual\footnote{In what follows, dual statements are often cited without pointing out that they are the duals of the original ones.}
of \init{G.\ }Cz\'edli, \init{L.~}Ozsv\'art, and  \init{B.~}Udvari \cite[Corollary 2.2]{czgolub}, and $\width{\Mir L}\leq 2$ by the definition of dual slimness.
Thus \eqref{sdszdFb} implies \eqref{sdszdFa}.
\end{proof}

Next,  motivated by Lemma~\ref{slimgeomlemma}, we will have a closer look at finite, dually slim, lower semimodular lattices. A finite lattice is \emph{planar} if it has a planar diagram in the obvious sense; for more details see the next subsection.   Planarity is a great help for us, because of Part (B) of the following lemma. A cover-preserving $M_3$ sublattice is a 5-element sublattice $\set{u,a_0,a_1,a_2,v}$ such that $u\prec a_i\prec v$ for $i\in\set{0,1,2}$.

\begin{lemma}[{\init{G.\ }Cz\'edli and \init{E.\,T.\ }Schmidt \cite[Lemmas 2.2 and 2.3]{czgschtJH}}]\label{slimplanar} 
If $L$ is a finite, lower semimodular lattice, then the following two statements hold.
\begin{itemize}
\item[(A)] $L$ is dually slim if{f} it has no cover-preserving $M_3$ sublattice. 
\item[(B)] If $L$ is dually slim, then it $L$ is planar, and each of its elements has at most two lower covers. 
\end{itemize}
\end{lemma}

\subsection{Dual slimness and  Carath\'eodory's  condition}
Following \init{L.\ }Libkin \cite{libkin}, a finite lattice $K$ is said to satisfy  \emph{Carath\'eodory's condition} \cprop n if for any $a,b_1,\dots,b_k\in \Jir K$  such that  $a\leq b_1\vee\dots\vee b_k$, there are $i_1,\dots,i_n\in\set{1,\dots,k}$ such that $a\leq b_{i_1}\vee\dots\vee b_{i_n}$.
Carath\'e\-odory's classical theorem asserts that whenever $p$ is a point and $X$ is a subset of $\mathbb R^{n-1}$ such that  $p$ belongs to the convex hull $\rhullo {n-1}(X)$ of $X$,   then there exists an at most $n$-element subset $Y$ of $X$ such that 
$p\in \rhullo {n-1}(Y)$. This implies that the lattices $\lat\pair E{\phullo {n-1}E}$ satisfy \cprop n{} since their join-irreducible elements are exactly the atoms. We say that a \emph{finite convex geometry 
$\pair U\Phi$ satisfies \cprop n{}} if so does the corresponding lattice, $\lat{\pair U\Phi}$. Lower semimodularity is not assumed in the following statement.

\begin{proposition}\label{cpropos}
Every finite, dually slim lattice satisfies \cprop 2.  Also, every finite convex geometry of convex dimension at most 2 satisfies \cprop 2. 
\end{proposition}

\begin{proof} The lattices in question are planar by \skc{{\init{G.\ }Cz\'edli and \init{E.\,T.\ }Schmidt \cite[Lemma 2.2]{czgschtJH}}},  and planar lattices satisfy \cprop 2{} by \init{L.\ }Libkin \cite[Corollary 4.7 and Theorem 3]{libkin}. The rest follows from Lemma~\ref{slimgeomlemma}.
\end{proof}

\begin{proposition}\label{jirmaprop} \refegy{Let $n\in \mathbb N$, and let  $L_1$ and $L_2$ be finite  lattices satisfying  \cprop n. 
If there exists a bijection $\phi\colon\Jir{L_1}\to \Jir{L_2}$ such that 
\begin{equation}\label{cp2djGt}
\text{for all $a,b_1,\dots b_n\in\Jir{L_1}$, 
$\,\,\,a\leq \bigvee_{i=1}^n b_i  \iff  \phi(a)\leq   \bigvee_{i=1}^n  \phi(b_i)$,}
\end{equation}
then} $\phi$ can be extended to an isomorphism from $L_1$ onto $L_2$.
\end{proposition}

\begin{proof} \refegy{First, we show that $n$ is not relevant in \eqref{cp2djGt}, that is,  \eqref{cp2djGt} implies the following property of $\phi$:
\begin{equation}\label{cpbmdjGt}
\text{for all $a\in\Jir{L_1}$ and $B\subseteq \Jir{L_1}$, 
$\,\,\,a\leq \bigvee B  \iff  \phi(a)\leq\bigvee \phi(B)$,}
\end{equation}
where $\phi(B)=\set{\phi(b):b\in B}$. Assume that $a\leq \bigvee B$. By \cprop n, there exists a subset $C\subseteq B$ such that $|C|\leq n$ and $a\leq\bigvee C$. Hence $ \phi(a)\leq\bigvee \phi(C)$ by \eqref{cp2djGt}, and $ \phi(a)\leq\bigvee \phi(B)$. The converse implication in \eqref{cpbmdjGt} is obtained by using $\phi^{-1}$. 
For the sake of Remark~\ref{latisorem} coming soon, 
we note that the rest of the proof relies only on \eqref{cpbmdjGt} and does not use any specific property of $L_1$ and $L_2$.}

Next, let $\psi\colon \Jir{L_2}\to\Jir{L_1}$ denote the inverse of $\phi$. We define a map $\ephi\colon L_1\to L_2$ by $\ephi(x)=\bigvee\set{\phi(a):a\in \ideal x\cap\Jir{L_1}}$.
Similarly, let $\epsi\colon L_2\to L_1$ be defined by $\epsi(y)=\bigvee\set{\psi(b):b\in \ideal y\cap\Jir{L_2}}$.
The choice \refegy{$b_1=\dots=b_n$} in \eqref{cp2djGt} shows that $\phi$ and $\psi$ are  order isomorphism. This implies that $\ephi$ and $\epsi$ are extensions of $\phi$ and $\psi$, respectively. \refegy{Using the formula $x=\bigvee(\ideal x\cap \Jir{L_i})$ for $x\in L_i$, it is routine to check that 
$\epsi$ and $\psi$ are reciprocal bijections.
Hence, they are lattice isomorphisms since they are  obviously order-preserving.} 
\end{proof}

\begin{remark}\label{latisorem}
 If $L_1$ and $L_2$ are finite lattices and $\phi\colon\Jir{L_1}\to \Jir{L_2}$ is a bijection satisfying \eqref{cpbmdjGt}, then $L_1\cong L_2$ and  $\phi$ extends to an isomorphism $L_1\to L_2$.
\end{remark}

\subsection{More about dually slim, lower semimodular lattices}
Even if $\phi$ in Proposition~\ref{jirmaprop} is an order-isomorphism,  \eqref{cp2djGt} \refegy{with $n=2$}  
may fail; its satisfaction depends mainly on the case where $\set{a,\refegy{b_1,b_2}}$ is an antichain. This is one of the reasons why we are going to have a closer look at dually slim, lower semimodular lattices.  Since dually slim lattices are planar by Lemma~\ref{slimplanar}(B), the propositions of this subsection  may look intuitively clear. However, their exact  proofs need some preparation.
Fortunately, the theory of planar lattices is satisfactorily developed in  \init{D.\ }Kelly and \init{I.\ }Rival \cite{kellyrival} at a rigorous level, so we can often rely on results from \cite{kellyrival} instead of going into painful rigorousity. Whenever we deal with a planar lattice, always a fixed planar diagram is assumed. Actually, most of the concepts, like left and right, depend on the planar diagram chosen (sometimes implicitly) at the beginning rather than on the lattice. This will not cause \skc{any} trouble  since our arguments do not depend on the choice of planar diagrams.

Now,  we recall some necessary concepts and statements for planar lattices; the reader may (but need not) look into \cite{kellyrival} for more exact details.
Let $C$ be a maximal chain  in a finite planar lattice $L$ (with a fixed planar diagram). This  chain cuts $L$ into
a \emph{left side} and a \emph{right side},  see \init{D.\ }Kelly and \init{I.\ }Rival \cite[Lemma 1.2]{kellyrival}. The
intersection of these sides is $C$. If $x\in L$ is on the left side of $C$ but not in $C$, then $x$ is \emph{strictly on the left} of $C$. Let $D$ be another maximal chain of $L$. If all elements of $D$ are on the left of $C$, then $D$ is \emph{on the left of} $C$. In this sense, we can speak of the leftmost maximal chain of $L$, called the \emph{left boundary chain}, and the rightmost maximal chain, called the \emph{right boundary chain}. The union of these two chains is the \emph{boundary} of $L$.
Also, if $E$ is a (not necessarily maximal) chain of $L$, then the \emph{leftmost maximal chain through} $E$ (or extending $E$) and the rightmost one make sense. 
If $E=\set{e_1<\dots <e_n}$, then the leftmost maximal chain of $L$ through $E$ is the union of the left boundary chains of the intervals $[0,e_1\refket]$, $[e_1,e_2]$, \dots, $[e_{n-1},e_n]$, and $[e_n,1]$. (The diagrams of these intervals are the respective subdiagrams of the fixed diagram of $L$.) If $E=\set e$ is a singleton, then  chains containing $e$ are said to be chains through $e$ rather than chains through $\set e$. 
The most frequently used result of \init{D.\ }Kelly and \init{I.\ }Rival  \cite{kellyrival} is probably the following one.

\begin{lemma}[{\init{D.\ }Kelly and \init{I.\ }Rival  \cite[Lemma 1.2]{kellyrival}}]
\label{krchainlemma}
Let $L$ be a finite planar lattice, and let $x\leq y\in L$. If $x$ and $y$ are on different sides of a maximal chain $C$ in $L$, then there exists an element $z\in C$ such that $x\leq z\leq y$.
\end{lemma}

Next, let $x$ and $y$ be elements of a finite planar lattice $L$, and assume that they are incomparable, in notation, $x\parallel y$. If $x\vee y$ has lower covers $x_1$ and $y_1$ such that $x\leq x_1\prec x\vee y$, $y\leq y_1\prec x\vee y$, and $x_1$ is on the left of $y_1$, then the element $x$ is \emph{on the left} of the element $y$. If $x$ is on the left of $y$, then we say that $y$ is \emph{on the right} of $x$. Let us emphasize that whenever left or right is used for two elements, then the elements in question are incomparable. 

\begin{lemma}[{\init{D.\ }Kelly and \init{I.\ }Rival  \cite[Propositions 1.6 and 1.7]{kellyrival}}]\label{leftrightlemma} Let $L$ be finite planar lattice\skc{. If}  $\,x,y\in L$ \skc{are} incomparable elements\skc{, t}hen the following hold.
\begin{itemize}
\item[(A)] $x$ is on the left of $y$ if{f} $x$ is on the left of some maximal chain through $y$ if{f} $x$ is on the left of all maximal chains through $y$.
\item[(B)] Either $x$ is on the left of $y$, or $x$ is on the right of $y$.
\item[(C)] If $z\in L$,  $x\parallel y$,  $y\parallel z$,  $x$ is on the left of $y$, and $y$ is on the left of $z$, then $x$ is on the left of $z$.
\end{itemize}  
\end{lemma}

If $\set{x_0,x_1,y}$ is a 3-element antichain such that  $x_i$ is on the left of $y$ and $y$ is on the left of $x_{1-i}$ for some (necessarily unique) $i\in \set{0,1}$, then $y$ is \emph{\skc{horizontally} between} the elements $x_0$ and $x_1$.

\begin{proposition}\label{propokoztes} Let 
$L$ be a finite lattice. \skc{If} $\set{x_0,x_1,y}$ is a 3-element antichain  in $L$\skc{, then} the following two statements hold. 
\begin{itemize}
\item[(A)]\label{lemmabetweenjoina} If $L$ is planar and 
$y$ is \skc{horizontally} between $x_0$ and $x_1$, then $y\leq x_0\vee x_1$.
\item[(B)]\label{lemmabetweenjoinb} If $L$ is a dually slim, lower semimodular lattice and $y\leq x_0\vee x_1$, then 
 $y$ is \skc{horizontally} between $x_0$ and $x_1$.
\end{itemize}
\end{proposition}

Note that it would  be unreasonable to tailor the same condition on $L$ in Parts (A) and (B). The 5-element, modular, non-distributive lattice $M_3$ indicates that planarity in itself would not be sufficient in Part (B). On the other hand, although dual slimness (with or without lower semimodularity) would be sufficient in Part (A) by  Lemma~\ref{slimplanar}, in this case the statement would be weaker and we could not use the dual of Part (A) in the proof of Part (B).

\begin{proof}[Proof of Proposition \ref{propokoztes}] 
To prove (A), pick a maximal chain $C$ through $y$. Without loss of generality,  we can assume that $x_0$ and $x_0\vee x_1$ are on the left and $x_1$ is on the right of $C$. Applying Lemma \ref{krchainlemma}, there exists an element $z\in C$ such that $x_1\leq z\leq x_0\vee x_1$. Belonging to the same chain, $y$ and $z$ are comparable. Since $z\leq y$ would contradict $x_1\not\leq y$, we have $y<z\leq  x_0\vee x_1$, proving (A).

Next, to prove (B) by contradiction, suppose that $L$ is a dually slim\skc{, lower semimodular lattice}, $y\leq x_0\vee x_1$, but $y$ is not between $x_0$ and $x_1$. Let, say, $x_0$ be on the left of $x_1$ and $x_1$ be on the left of $y$. 
The interval $I=[x_0\wedge y, x_0\vee x_1]$ contains $x_0$, $x_1$ and $y$ since $y\leq x_0\vee x_1$ and, by the dual of Part (A), $x_0\wedge y\leq x_1$. As an interval of $L$, it is lower semimodular, and it follows from Lemma~\ref{slimplanar}(A) that this interval is dually slim. There are two cases.

First, assume $y\vee x_1=x_0\vee x_1$. Let $E$ be a maximal chain of $I$ through $\set{x_1, x_0\vee x_1}$. By Lemma~\ref{leftrightlemma}(A), $x_0$ is on the left of $E$ and $y$ is on the right of $E$. Note that the left side of $E$ (including $E$ itself) is a cover-preserving sublattice by \init{D.\ }Kelly and \init{I.\ }Rival \cite[Proposition 1.4]{kellyrival}; this can also be derived from  Lemma~\ref{krchainlemma} easily. Hence, 
we can pick lower covers $x_0'$, $x_1'$ of $x_0\vee x_1$ on the left of $E$
such that $x_0\leq x_0'$ and $x_1\leq x_1'$. Similarly, let $y'\in I$ be a lower cover of $x_0\vee x_1$ such that $y\leq y'$ and $y'$ is on the right of $E$. 
If we had $x_0'=y'$, then $E$ would contain an element $z$ by Lemma~\ref{krchainlemma}, necessarily strictly above $x_1$ since  $y\not\leq x_1$, such that $y\leq z\leq x_0'$, which would lead to the following contradiction:
\[\text{$y\vee x_1\leq z\leq x_0'\prec x_0\vee x_1=y\vee x_1$.}
\] 
Hence, $x_0'\neq y'$. If we had $x_0'=x_1'$ or $x_1'=y'$, then $x_0\vee x_1\leq x_1'\prec x_0\vee x_1$ or 
$y\vee x_1\leq x_1'\prec x_0\vee x_1=y\vee x_1$, respectively, would be a contradiction. Hence $x_0'$, $x_1'$ and $y'$ are three distinct lower covers of $x_0\vee x_1$, which contradicts Lemma~\ref{slimplanar}(B).

Second, assume $y\vee x_1\neq x_0\vee x_1$. Since these elements are in $I$, we have $y<y\vee x_1<x_0\vee x_1$. 
Take a maximal chain $G$ of $I$ through $\set{y, y\vee x_1}$, and let $J$ be the left side of $G$. By Lemma~\ref{leftrightlemma}(A), $x_0,x_1\in J$. In $J$, take a maximal chain $F$ through $x_0$, and let $K$ be the right side of $F$ in $J$. By Lemma~\ref{leftrightlemma}(A) again, $x_1\in K$. 
\skc{Clearly,} $K$ is a cover-preserving sublattice of $I$, again by \init{D.\ }Kelly and \init{I.\ }Rival  \cite[Propositions 1.4]{kellyrival}. 
Now $F$ and $G$ are the left and  right boundary chains of $K$, respectively. Like we obtained it for $I$, we conclude that $K$, which is a cover-preserving sublattice of $I$ (and also of $L$),  is a dually slim, lower semimodular lattice. By \init{G.\ }Cz\'edli and \init{E.\,T.\ }Schmidt \cite[Lemma 6]{czgschvisual}, $\Mir K\subseteq F\cup G$. Therefore, there are $f\in F$ and $g\in G$ such that $x_1=f\wedge g$. 
Since $f\leq x_0$ is excluded by $x_1\not\leq x_0$ and $F$ is a chain, $x_0< f$. Similarly, using that $G$ is a chain and $y\parallel x_1$, we obtain  $y < g$. 
Thus  $y\vee x_1\leq g$, and we conclude $x_1=f\wedge (y\vee x_1)$.
Clearly, $f\neq x_0\vee x_1$ since otherwise $x_1=(x_0\vee x_1)\wedge (y\vee x_1)= y\vee x_1 \geq y$ would contradict $x_1\parallel y$. Hence $f< x_0\vee x_1$, which leads to the contradiction 
$x_0\vee x_1\leq f<x_0\vee x_1$. 
\end{proof}

We also need the following statement.

\begin{proposition}\label{prszdobzF}
If $L$ \refegy{is} a finite  lower semimodular lattice, $a\in \Jir L$, $b,c\in L$, $c<a$, and $a\leq b\vee c$, then $a\leq b$.
\end{proposition}

\begin{proof} To prove the statement by contradiction, suppose that in spite of the assumptions, $a\not\leq b$. 
If we had $b<a$, then  $a\geq b\vee c$ together with  $a\leq b\vee c$ would contradict $a\in\Jir L$. Hence $b\parallel a$. Let $d=b\vee c$; we know that $a\leq d$. Since $b
\nonparallel c$ together with $b\vee c\geq a$ would contradict $\set{b,c}\cap\filter a=\emptyset$, we have $b\parallel c$. Hence, $b<d$ and we can pick an element $e\in [b,d]$ such that $e\prec d$.
Denoting the unique lower cover of $a$ by $\lowstar a$,  we conclude
\[ d=e\vee a=e\vee\lowstar a,\quad e\parallel a,\quad\text{and}\quad e\parallel \lowstar a,
\]
because of the following reasons: $c\leq \lowstar a$ yields 
$d=b\vee c\leq e\vee \lowstar a\leq e\vee a\leq d$; $e\leq \lowstar a$ or $e\leq a$ would contradict $b\parallel a$; and $e\geq \lowstar a$ or $e\geq  a$ would lead to the contradiction $e\prec d= e\vee\lowstar a =e$ or $e\prec d= e\vee a =e$.

Since $e\parallel a$, we have $e\wedge a <a$. Lower semimodularity yields $e\wedge a\preceq d\wedge a=a$, and we obtain $e\wedge a\prec a$. Since $\lowstar a\parallel e$ and $e\wedge a \leq e$, the elements $\lowstar a$ and $e\wedge a$ are two distinct lower covers of $a$. 
This contradicts $a\in\Jir L$.
\end{proof}

The elements of $\Jir L\cap\Mir L$ are called \emph{doubly irreducible elements}.
A principal filter $\filter b$ of $L$ is a \emph{prime filter} if $\skc{\emptyset} \neq L\setminus\filter b$ is closed with respect to joins or, equivalently, if $L\setminus\filter b$ is a lattice ideal of $L$.

\begin{proposition}\label{dirronbtop} Let $L$ be a finite, dually slim, lower semimodular lattice\skc{. If}  $|L|\ge 3$, then the following \skc{three} statements hold.
\begin{enumeratei}
\item\label{dirronbtopa} $L$ has a maximal doubly irreducible element $b$, and \skc{this} $b$ belongs to the boundary of $L$. 
\item\label{dirronbtopb} If $x\in L$ and $x>b$, then $x\in\Mir L$ but $x\notin \Jir L$. Furthermore, $\filter b$ is a chain.
\item\label{dirronbtopc} $\filter b$ is a prime filter of $L$.
\end{enumeratei}
\end{proposition}

\begin{proof} By Lemma~\ref{slimplanar} (B), $L$ is planar. We know from 
\init{D.\ }Kelly and \init{I.\ }Rival  \cite[Theorem 2.5]{kellyrival} that each finite planar lattice $L$ with at least three elements has a doubly irreducible element on its boundary. However, we only use this theorem to conclude 
$\Jir L\cap\Mir L\neq\emptyset$.
Hence, $\Jir L\cap\Mir L$ contains a maximal element, $b$. Since $\Mir L$ is a subset of the boundary by the dual of \init{G.\ }Cz\'edli and \init{E.\,T.\ }Schmidt \cite[Lemma ~6]{czgschvisual}, $b$ belongs to the boundary. This proves \eqref{dirronbtopa}.

Let, say, $b$ belong to the right boundary chain $C_r$ of $L$. The dual of \init{G.\ }Cz\'edli \cite[Lemma 2.3]{czgasympt} asserts $C_r\cap\filter b\subseteq \Mir L$. We claim that $\filter b\subseteq C_r$. To prove this by contradiction, suppose 
$\filter b\not\subseteq C_r$. \skc{We obtain that} $\filter b$ is not a chain  since $\filter b\cap C_r$ is a maximal chain in $\filter b$. Hence, there are $u,v\in\filter b$ such that $u\parallel v$, and there are chains $b=u_0\prec u_1\prec\dots\prec u_t=u$ and $b=v_0\prec v_1\prec\dots\prec v_s=v$ in $\filter b$. Clearly, $s,t\geq 1$. Let $i$ be the largest subscript such that $i\leq t$, $i\leq s$, and  $u_i=v_i\in C_r$. This $i$ exists since $u_0=v_0=b\in C_r$. 
Since $u\parallel v$, we have $i< s$ and $i< t$. There are two cases. First, if $u_{i+1}\neq v_{i+1}$, then $u_i=v_i$ has at least two distinct covers. Second, if $u_{i+1}= v_{i+1}\notin C_r$, then $u_i=v_i$ has at least two distinct covers again:  $u_{i+1}= v_{i+1}$ and a cover belonging to $C_r$. Hence, in both cases, $u_i=v_i$ is meet-reducible and belongs to $C_r$, which contradicts  $C_r\cap\filter b\subseteq \Mir L$. Therefore, $\filter b\subseteq C_r$ and, since subsets of chains are chains, $\filter b$ is a chain. 

Now, assume $x>b$. Since $x\in \filter b=C_r\cap \filter b\subseteq \Mir L$ and $b$ was a maximal doubly irreducible element, we conclude  $x\notin\Jir L$. This proves \eqref{dirronbtopb}.

To prove \eqref{dirronbtopc} by contradiction, suppose $y\in \filter b$ such that there exist elements $u,v\in  L\setminus \filter b$ with $y=u\vee v$. 
\skc{We have} $u=u_1\vee \dots \vee u_s$ and $v=v_1\vee\dots \vee  v_t$ for some $u_1,\dots, u_s,v_1,\dots,v_t\in \Jir L\setminus \filter b$. It follows from Proposition~\ref{cpropos} that there are two elements in $\set{u_1,\dots, u_s,v_1,\dots,v_t}$ whose join belongs to $\filter b$. Therefore, there are incomparable elements $p,q\in \Jir L\setminus \filter b$ such that $b\leq p\vee q$.  There are two cases. First, assume that $\set{p,q,b}$ is an antichain. \skc{Clearly,} none of $p$ and $q$ belongs to $C_r$, and Lemma~\ref{leftrightlemma}(A) yields that both $p$ and $q$ are on the left of $b$. This is a contradiction since Proposition~\ref{propokoztes}(B) implies that $b$ is \skc{horizontally} between $p$ and $q$.

Second, assume that $\set{p,q,b}$ is not an antichain. Since $b$ is join-irreducible and $p,q\notin \filter b$, we cannot have $\set{p,q}\subseteq \ideal b$. Hence, apart from $p$-$q$ symmetry, $p\parallel b$ and $q<b$. However, \skc{now} Proposition~\ref{prszdobzF} contradicts $b\leq p\vee q$.
\end{proof}

\subsection{The rest of the proof} Before formulating the last auxiliary statement towards  Theorem~\ref{thmmain}, remember that dual slimness implies planarity by Lemma~\ref{slimplanar}(B).

\begin{lemma}\label{lMlStm} Let $L$ be a finite, dually slim, semimodular lattice with a fixed planar diagram, and let $F$ be  a finite concave set of collinear circles in the plane. Assume that we have a bijective map $\psi\colon \Jir L\to F$ such that for any $u,v\in \Jir L$,  
\begin{enumeratei}
\item\label{lMlStma} $u\leq v$ if and only if $\psi(u)\subseteq \rhullo2(\psi(v))$, and
\item\label{lMlStmb}  $u\parallel v$ and $u$ is on the left of $v$ if and only if $\,\lend{\psi(u)}\sqsubset \lend{\psi(v)}$ and $\rend{\psi(u)}\sqsubset \rend{\psi(v)}$.
\end{enumeratei}
\skc{These assumptions imply} $\,L\cong \lat{\pair F{\chullo F}}$. 
\end{lemma}

\begin{proof} Lemma~\ref{mzrrDcgbn} allows us to define a \skc{bijective} map $\phi\colon \Jir L\to \Jir{\!\bigl(\lat{\pair F{\chullo F}}\bigr)}$ by $\phi(u)=\fromto F{\psi(u)}{\psi(u)}$.
Clearly, 
\[\psi(u)\subseteq \rhullo 2 (\psi(v))\,\, \iff \,\,\fromto F{\psi(u)}{\psi(u)} \subseteq \fromto F{\psi(v)}{\psi(v)}\text.
\]
Therefore, by Assumption \eqref{lMlStma}, $\phi$ is an order-isomorphism.
Since we want to apply Propositions~\refegy{\ref{cpropos} and} \ref{jirmaprop}, we are going to show that $\phi$ satisfies Condition \eqref{cp2djGt} \refegy{with $\tuple{2,a,b,c}$ in place of $\tuple{n,a,b_1,\dots,b_n}$}.

If $a\leq b$ or $a\leq c$, then $\phi(a)\leq \phi(b)$ or $\phi(a)\leq \phi(c)$ since $\phi$ is an order-isomorphism\skc{. The case $b\nonparallel c$ is even more evident.} Thus, if $\set{a,b,c}$ is not an antichain,  we have
\begin{equation}\label{aleqbveec}
a\leq b\vee c \iff \phi(a)\leq \phi(b)\vee \phi(c)\text. 
\end{equation}

Next, assume that $\set{a,b,c}$ is an antichain. \skc So is $\set{\phi(a),\phi(b),\phi(c)}$ since $\phi$ is an order-isomorphism.  There are two cases: either $a$ is \skc{horizontally} between $b$ and $c$, or not. 

In the first case, we can assume that $b$ 
is on the left of $a$ and $a$ is on the left of $c$. \skc Proposition~\ref{propokoztes} gives $a\leq b\vee c$, and 
\eqref{lMlStmb} yields 
\begin{equation}\label{dzVbXy}
\begin{aligned}
&\lend{\psi(b)}\sqsubset \lend{\psi(a)}<\lend{\psi(c)}\,\,\text{ and}\cr
&\rend{\psi(b)}<\rend{\psi(a)}\sqsubset \rend{\psi(c)}
\end{aligned}
\end{equation}
Since $F$ is concave, \eqref{dzVbXy} implies $\psi(a)\subseteq \rhullo2 \bigl(  {\psi(b)}\cup {\psi(c)} \bigr)$, which gives $\phi(a)\leq \phi(b)\vee \phi(c)$. Hence, \eqref{aleqbveec} holds in this case.

In the second case, where $a$ is not \skc{horizontally} between $b$ and $c$, we can assume that $a$ is on the left of $b$ and $b$ is on the left of $c$. \skc By 
Proposition~\ref{propokoztes}, we have $a\not\leq b\vee c$, and \eqref{lMlStmb} gives
\begin{equation}\label{dznTsXy}
\begin{aligned}
&\lend{\psi(a)}\sqsubset \lend{\psi(b)}\sqsubset \lend{\psi(c)}\,\,\text{ and}\cr
&\rend{\psi(a)}\sqsubset \rend{\psi(b)}\sqsubset \rend{\psi(c)}
\end{aligned}
\end{equation}
Let $X=\fromto F{\psi(b\skc )}{\psi(c)}$. \skc{We have}  
$\set{\psi(b),\psi(c)}\subseteq X\in \lat{\pair F{\chullo F}}$ by Lemma \ref{fromtolemma}, and $\psi(a)\notin X$ by the definition of horizontal intervals. Since $\phi(b)\vee\phi(c)\subseteq X$, we conclude $\phi(a)\not\leq \phi(b)\vee \phi(c)$, and \eqref{aleqbveec} holds in this case.

Next, assume $a>b$ and $a>c$, and let $\lowstar a$ stand for the unique lower cover of $a$. \skc{Now} $\lowstar a\geq b\vee c$, and we have $a\not\leq b\vee c$. Since $\psi$ is an order-isomorphism, $\phi(\lowstar a)\geq\phi(b)\vee\phi(c)$.  This gives $\phi(a)\not\leq \phi(b)\vee\phi(c)$, and \eqref{aleqbveec} is fulfilled again.

Finally, up to $b$-$c$ symmetry, we are left with the case where $b\parallel a$ and $c<a$. We can assume $b\parallel c$ since otherwise both sides of \eqref{aleqbveec} would obviously be false and \eqref{aleqbveec} would hold.  Take a maximal chain $C$ including $\set{c,a}$; it does not contain $b$. Let, say, $b$ be on the left of $C$. Now, by Lemma~\ref{leftrightlemma}, $b$ is on the left of $c$ and also on the left of $a$.  
Since $\psi(c)\mathrel{\skc{\subseteq}} \rhullo2(\psi(a))$ \skc{by Assumption \eqref{lMlStma}}, 
we conclude  $\rend{\psi(c)}\mathrel{\skc{\sqsubseteq}} \rend{\psi(a)}$. \skc{Thus $\rend{\psi(c)}\sqsubset \rend{\psi(a)}$ since our circles are determined by their left ends.}
%
This and \eqref{lMlStmb} yield
\begin{equation}\label{dzWTZX}
\begin{aligned}
&\lend{\psi(b)}\sqsubset \lend{\psi(c)}\,\,\text{ and }\,\,\cr 
&\rend{\psi(b)}\sqsubset \rend{\psi(c)}\sqsubset \rend{\psi(a)}\text.
\end{aligned}
\end{equation}
As previously, this gives $\set{\psi(b),\psi(c)}\subseteq \fromto F{\psi(b)}{\psi(c)}\in \lat{\pair F{\chullo F}}$ and 
$\psi(a)\notin\fromto F{\psi(b)}{\psi(c)}$, which implies $\phi(a)\not\leq\phi(b)\vee\phi(c)$. Since $a\not\leq b\vee c$ by Proposition~\ref{prszdobzF}, \eqref{aleqbveec} is satisfied again.

Since \eqref{aleqbveec} holds in all cases, $\phi$ satisfies \eqref{cp2djGt}. Thus Propositions~\skc{\ref{cpropos} and} \ref{jirmaprop} appl\skc{y}.
\end{proof}

\begin{proof}[Proof of Theorem~\ref{thmmain}] 
Let  $F$ be a finite, concave set of  collinear circles in the plane. Proposition~\ref{prop1} yields that $\pair F{\chullo F}$
is a convex geometry. Hence, by Lemma~\ref{geomversuslat},  $\lat{\pair F{\chullo F}}$ is a finite meet-distributive lattice, and it is lower semimodular by Claim \ref{mdstrthnlsm}.
We obtain from Lemma~\ref{mirlemma} that this lattice is dually slim. Therefore, Lemma~\ref{slimgeomlemma} implies that $\pair F{\chullo F}$
is a convex geometry of convex dimension at most 2. This proves part (A).

In view of Lemmas~\ref{geomversuslat} and \ref{slimgeomlemma}, Part (B) is equivalent to the following statement:
\begin{enumeratei}
\item[(C)] If $L$ is a finite, dually slim, lower semimodular lattice, then there exists a finite,  separated, concave set $F$ of  collinear circles in the plane such that $L$ is isomorphic to $\lat{\pair F{\chullo F}}$.
\end{enumeratei}
We prove (C) by inductio\skc n.
By Lemma~\ref{lMlStm}, it suffices to construct a pair $\pair F\psi$ such that $F$ is a finite,  separated, concave se\skc{t of}  collinear circles in the plane and $\psi\colon \Jir L\to F$ is a bijective map satisfying Conditions \eqref{lMlStma} and \eqref{lMlStmb} of Lemma~\ref{lMlStm}. 
Since we are going to construct a separated $F$,  \eqref{obvruuules} allows us to satisfy these two conditions with leftmost and rightmost endpoints rather \refket{than} left and right ends. 

First, assume that  $L$ is a chain or, equivalently,  $\Jir L$ is a chain.  We can let $F$ and $\psi$ be a set of concentric circles and  the unique map satisfying Condition \eqref{lMlStma} of Lemma~\ref{lMlStm}, respectively; clearly, $\pair F\psi$ is an appropriate pair.  More generally, not assuming that $L$ is a chain, we can prove the existence of an appropriate pair $\pair F\psi$ by induction on the size $|L|$ of $L$. Since the case of chains has been settled, the induction starts at size 4.

\begin{figure}
\centerline
{\includegraphics[scale=1.0]{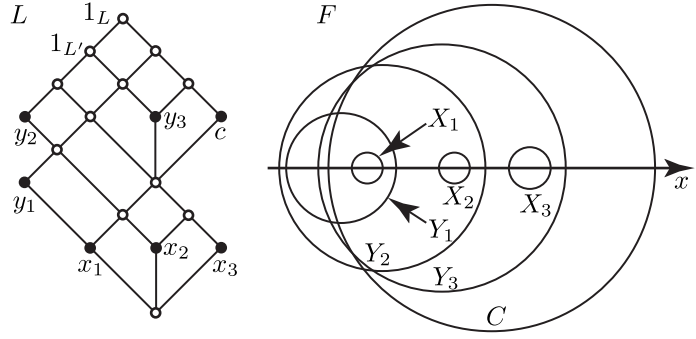}}
\caption{A dually slim lower semimodular lattice $L$ and the corresponding separated concave set $F$ of collinear circles\label{fig3}}
\end{figure}

Assume that $|L|\geq 4$ and   for each finite, dually slim, lower semimodular lattice of smaller size, there exists and appropriate pair. Take a maximal doubly irreducible element $c\in L$ \skc{(like $b$ in} Proposition~\ref{dirronbtop}\skc). By left-right symmetry, we can assume that $c$ is on the right boundary chain of $L$;
see Figure~\ref{fig3}, on the left, for illustration. (The figure serves only as an  illustration, so the reader need not check the properties of $L$. However, we note that $L$ is obviously a dually slim, lower semimodular lattice by the dual of \init{G.~}Cz\'edli and \init{E.\,T.\ }Schmidt \cite[Theorem 12]{czgschvisual}.)
Since $\filter c$ is a prime filter \skc{by Proposition~\ref{dirronbtop}}, $L'=L\setminus \filter c$ is \refket{an} ideal of $L$. As an interval of $L$, $L'$ is lower semimodular. It follows from Lemma~\ref{slimplanar}(A) that $L'$ is dually slim. Hence, by the induction hypothesis, there exists an appropriate pair $\pair{F'}{\psi'}$ for $L'$. We want to define $\pair F\psi$ such that $F'\subset F$ and $\psi$ be an extension of $\psi'$. By Proposition~\ref{dirronbtop}\eqref{dirronbtopb}, $\Jir L=\Jir{L'}\cup\set c$. Therefore, our purpose is to find and appropriate circle $C$ and to let $\psi$ be the map from $\Jir L$ to $F=F'\cup\set C$ defined by $\psi(c)=C$ and $\psi(x)=\psi'(x)$ for $x\in \Jir{L'}$.
In the figure, $\Jir{L'}=\set{x_1,x_2,x_3,y_1,y_2,y_3}$ and $\Jir L= \Jir{L'}\cup\set c$ is the set of black-filled elements. On the right of the figure, 
we write $C$, $X_i$, and $Y_i$ instead of $\psi(c)$, $\psi'(x_i)$, and $\psi'(y_i)$, respectively;  $F$ is the collection of all circles, and $F'=F\setminus\set C$. 

\begin{figure}
\centerline
{\includegraphics[scale=1.0]{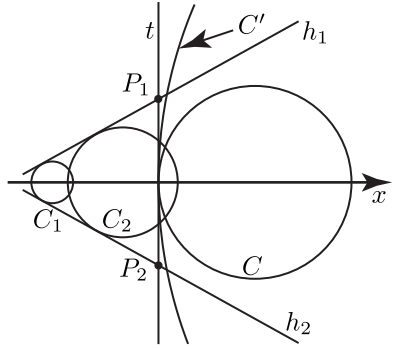}}
\caption{Modifying $C$ to make $F$ concave \label{fig4}}
\end{figure}

Since $c$ is on the right boundary chain, Lemma~\ref{leftrightlemma}(A)  yields that each $x\in \Jir L$ is either strictly on the left of $c$ or comparable to $c$. This together with  
Proposition~\ref{dirronbtop}\eqref{dirronbtopb} imply that $\Jir{L'}$ is the disjoint union of the following two sets:
\[\Jbelow=\Jir{L'}\cap\ideal c   \,\,\text{ and }\,\,\Jleft=\set{y\in\Jir{L'}: y\text{ is on the left of }c} \text.
\]
Note that one of these two sets can be empty but their union,  $\Jir{L'}$, is nonempty. In the figure, $\Jbelow=\set{x_1,x_2,x_3}$ and $\Jleft=\set{y_1,y_2,y_3}$.
Condition \eqref{lMlStma} of Lemma~\ref{lMlStm} will hold if{f}
\begin{equation}\label{lmpchoa}
\text{$\lmpt C < \lmpt{\psi'(x)}$ 
and $\rmpt{\psi'(x)} < \rmpt C$ 
for all $x\in\Jbelow$.} 
\end{equation}
Similarly, Condition \eqref{lMlStmb} of Lemma~\ref{lMlStm} will hold if{f}
\begin{equation}\label{lmpchob}
\text{$\lmpt{\psi'(y)} < \lmpt C$ 
and $\rmpt{\psi'(y)} < \rmpt C$ 
for all $y\in\Jleft$.} 
\end{equation}
The only stipulation that \eqref{lmpchoa} and \eqref{lmpchob} tailor on  $\rmpt C$ is  
\begin{equation}\label{rmptflfrg}
\rmpt C>\max\set{\rmpt{\psi'(z)}:z\in\Jir{L'} };
\end{equation}
this can be satisfied easily. 
Therefore, to see that we can choose $\lmpt C$ such that \eqref{lmpchoa} and \eqref{lmpchob} hold, it suffices  to show that 
\begin{equation}\label{szBjHks}
\text{for all }x\in\Jbelow\text{ and }
y\in \Jleft\text{, }\lmpt{\psi'(y)}<\lmpt{\psi'(x)}\text.
\end{equation}
For $x\in \Jbelow$ and $y\in \Jleft$,  there are two cases. 
First, assume  $x\nonparallel y$. Clearly, $y\not\leq x$, thus $x<y$. Since Condition \eqref{lMlStma} of Lemma~\ref{lMlStm} holds for $\pair{F'}{\psi'}$, we have $\lend {\psi(y)}\sqsubset \lend {\psi(x)}$ and $\rend {\psi(x)}\sqsubset \rend {\psi(y)}$.  This yields $\lmpt{\psi'(y)}<\lmpt{\psi'(x)}$ since $F'$ is separated.  
Second, assume $x\parallel y$. 
Let $E$ be a maximal chain in $L$ that extends $\set{x,c}$. Since $y$ is on the left of $c$, $y$ is on the left of $E$ by Lemma~\ref{leftrightlemma}(A). Hence, again by Lemma~\ref{leftrightlemma}(A), $y$ is on the left of $x$. 
Therefore, using 
Condition \eqref{lMlStmb} of Lemma~\ref{lMlStm} for the appropriate pair
$\pair{F'}{\psi'}$, we conclude $\lend{\psi'(y)}\sqsubset  \lend{\psi'(x)}$. This implies $\lmpt{\psi'(y)}<\lmpt{\psi'(x)}$ since $F$ is separated. Thus \eqref{szBjHks} holds, and so do \eqref{lmpchoa} and \eqref{lmpchob}.

Since \eqref{lmpchoa} and \eqref{lmpchob} are strict inequalities, we can choose both $\lmpt C$ and $\rmpt C$ infinitely many ways. Therefore, we can choose $C$ so that $F$ be separated.

Finally, we have to show that $F$ is concave. Since $F'$ is concave and separated, the only  case we have to consider is 
\begin{equation}\label{ksjsjS}
\text{$\lmpt{C_1} < \lmpt{C_2}$ and, automatically, $\rmpt{C_2} < \rmpt{C}$,}
\end{equation}
where $C_1,C_2\in F'$; 
\skc{we}  have to show 
\begin{equation}\label{sjGlSGW}
C_2\subseteq \rhullo2( C_1\cup C)\text. 
\end{equation}
Suppose that after choosing $C$, \eqref{sjGlSGW} fails for some $C_1,C_2\in F'$, see Figure~\ref{fig4}. \skc{The circle} $C$ is in the interior of the region between the two common tangent lines $h_1$ and $h_2$ of $C_1$ and $C_2$. Let $t$ be the tangent line of $C$ through $\pair{\lmpt C}{0}$, and let $P_i$ be the intersection point of $h_i$ and $t$ for $i\in\set{1,2}$. If  $\rmpt C$ tends to infinity while $\lmpt C$ is unchanged, then the arc of $C$  between $h_1$ and $h_2$ with middle point $\pair{\lmpt C}0$  approaches the line segment $P_1P_2$. Therefore, replacing $C$ by $C'$ such that $\lmpt {C'}=\lmpt C$ and $\rmpt{C'}$ is sufficiently large, we have $C_2\subseteq \rhullo2 ( C_1\cup C' )$ and \eqref{rmptflfrg}. We can treat all pairs $\pair{C_1}{C_2}\in F'\times F'$ with $\lmpt{C_1}<\lmpt{C_2}$,   each after each, because $\rmpt C$ can always be enlarged. This proves that $F$ is concave for some $C$.
\end{proof}

\section{Odds and ends}\label{oddsendch}
\begin{remark}
It is not hard to see that Theorem~\ref{thmmain} remains valid if we consider closed discs or open discs instead of circles, and modify the definitions accordingly. The advantage of  circles is that they are easier to visualize and label in figures. 
Open discs  are particularly less pleasant than circles since they cannot be singletons.
Note that we cannot use semicircles or half discs since, by the following example, the corresponding structure is not a convex geometry in general.
\end{remark}

\begin{figure}
\centerline
{\includegraphics[scale=1.0]{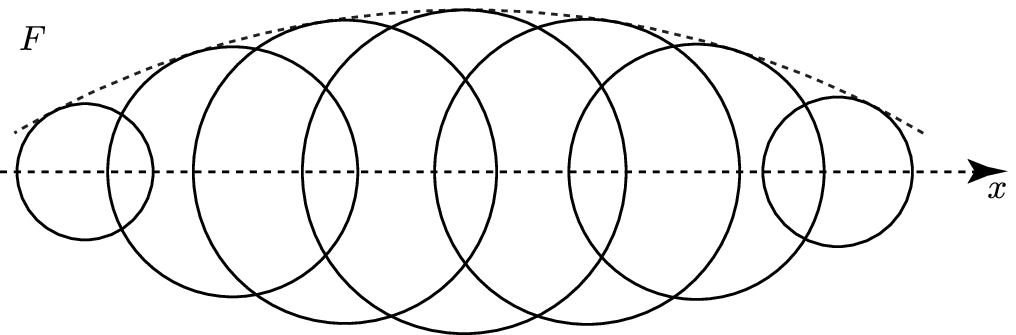}}
\caption{Here  $\lat{\pair F{\chullo F}}$ is the 128-element boolean lattice\label{fig5}}
\end{figure}

\begin{example} Let 
\[H_1=\set{\pair xy: x^2+y^2=4\text{ and }x\leq 0},\text{ } H_2=\set{\pair xy: x^2+y^2=1\text{ and }x\geq 0}\text.
\]
Rotating $H_2$ around $\pair 00$ by angle $\pi/100$, we obtain a half circle $H_3$. Since $H_{5-i}$ \skc{belongs to} $\rhullo2(H_1\cup H_i)$ for $i\in\set{2,3}$, the anti-exchange 
property fails, and we do not obtain a convex geometry from $\set{H_1,H_2,H_3}$.
\end{example}

\begin{figure}
\centerline
{\includegraphics[scale=1.0]{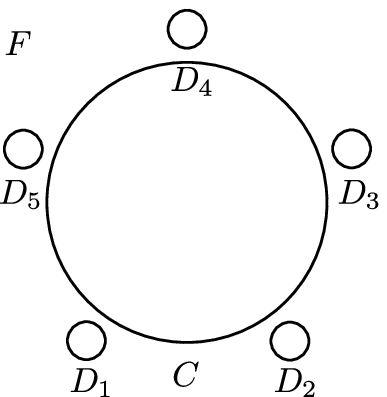}}
\caption{An $\pair F{\chullo F}$ that does not satisfy \cprop 4  \label{fig6}}
\end{figure}

The following example show\skc s that the convex dimension of  $\pair F{\chullo F}$ can be arbitrarily large even if the circles in $F$ are collinear. 

\begin{example}\label{krkrntNk} Let $F$ be an $n$-element set of collinear circles. Assume that there is an additional circle $K$ such that every circle  $C\in F$ is internally tangent to $K$; see Figure~\ref{fig5} for $n=7$, where the dotted curve is an arc of $K$.
\skc{Clearly,} $\lat{\pair F{\chullo F}}=\pair{\powset F}{\subseteq}$ is the $2^n$-element boolean lattice, and the convex dimension of $\lat{\pair F{\chullo F}}$ is $n$.
\end{example}

\begin{proof} The equality $\lat{\pair F{\chullo F}}=\pair{\powset F}{\subseteq}$ is obvious. Since 
\[\Mir\pair{\powset F}{\subseteq} = \set{F\setminus\set C:C\in F}\] 
is an $n$-element antichain, the convex dimension is $n$.
\end{proof} 

While $\pair E{\phullo 2E}$ satisfies  Carath\'eodory's condition \cprop 3 for every finite set $E$ of points of the \refegy{plane}, the following example shows that 
circles are essentially different from points. 

\begin{example}\label{cvmdfPnTs} For each natural number $n$, there exists an $(n+2)$-element set $F$ of circles in the \refegy{plane} such that $\pair F{\chullo F}$ does not satisfy \cprop n. For example, we can take the inscribed circle of a regular $(n+1)$-gon and $n+1$ additional little circles whose centers are the vertices of the $(n+1)$-gon; see Figure~\ref{fig6} for $n=4$.
\end{example}

This example has no collinear counterpart since we have the following proposition.

\begin{proposition}\label{prsoZHrW}
 If $F$ is a finite set of collinear circles, then ${\pair F{\chullo F}}$ satisfies  Carath\'eodory's condition  \cprop 2. 
\end{proposition}

\begin{proof} 
In view of Lemma~\ref{mzrrDcgbn}, we have to show the following: if $C,D_1,\dots,D_k\in F$ such that $C\subseteq \rhullo2(\refket{D_1\cup\dots \cup D_k } )$, then there exist $i,j\in\set{1,\dots,k}$ such that $C\subseteq \rhullo2(D_i\cup D_j)$. Let $G$ be the boundary of $\rhullo2(D_1\cup\dots D_k)$\skc; see the thick closed curve in Figure~\ref{fig7}, where $k=5$ and $F$ contains  the solid circles and possibly some other circles not indicated. (The dotted circle need not belong to $F$.) Clearly, $G$ can be divided into circular arcs and  straight line segments of common tangent lines of some circles belonging to $\set{D_1,\dots,D_k}$; these parts are separated by black-filled points in the figure. Keeping its center fixed, we enlarge $C$ to $C'$ such that $C'\subseteq \rhullo2(D_1\cup\dots D_k)$ and $C'$ is internally tangent to $G$ at a point $T\in G\cap C'$. There are two cases.

\begin{figure}
\centerline
{\includegraphics[scale=1.0]{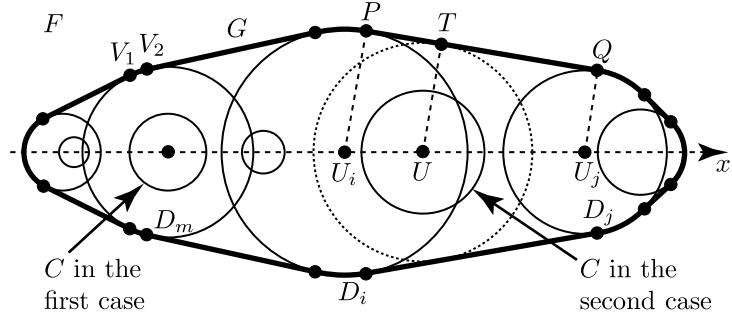}}
\caption{Illustrating the proof of Proposition~\ref{prsoZHrW}  \label{fig7}}
\end{figure}

In the first case, we assume that $T$ belongs to a  circular arc of $G$. In the figure, $T$ in this case is not indicated; it can be any point of the closed circular arc $V_1V_2$. 
This arc is also an arc of some member, $D_m$, of $F$. Clearly, $C'=D_m$. Thus  $D_m$ and $C$ are concentric circles, and $C\subseteq \rhullo2(D_m)$. Hence, we can let $i=m$ and $j=m$. Observe that  circular arcs of length 0 cause no problem since then the radius of $D_m$ is zero, the center of $D_m$ belongs to $G$, and $C=D_m$.

In the second case, we assume that $T$ belongs to a  straight line segment $PQ$ of $G$. Now $C'$ is the dotted circle in the figure, and its center is $U$. We can assume that $T\notin\set{P,Q}$ since otherwise the previous case applies. Clearly, $P$ is a point of a unique $D_i$ with center $U_i$,  and $Q$ is on a unique $D_j$ with center $U_j$.  Since the radii $PU_i$, $TU$ and $QU_j$ are all perpendicular to the  common tangent line $PQ$, it follows that $U$ is between $U_i$ and $U_j$, and $C'\subseteq\rhullo2(D_i\cup D_j)$. Therefore, $C\subseteq\rhullo2(D_i\cup D_j)$.
\end{proof}

For a class $\mathcal U$ of structures, let $\izom{\mathcal U}$ denote the class of structures that are isomorphic to some members of $\mathcal U$. We consider the following classes of finite convex geometries; by circles we mean circles in the \refegy{plane}.
{\allowdisplaybreaks{
\begin{align*}
\kccc&=\izom{\set{{\pair F{\chullo F}}: F\text{ is  a finite, concave set of collinear circles}    }},\cr
\kccol&=\izom{\set{{\pair F{\chullo F}}: F\text{ is  a finite set  of collinear circles}    }},\cr
\kcplan&=\izom{\set{{\pair F{\chullo F}}: F\text{ is  a finite set  of  circles}    }},\cr
\prel n&=\izom{ \set{{\pair E{\phullo nE}}: E\text{ is a finite subset of } \mathbb R^{n }} },\cr
\kcall&=\text{the class of all finite convex geometries.}
\end{align*}
}}%

Results by \init{K.~}Adaricheva \cite{adaricheva} and \init{G.\,M.}Bergman~\cite{bergman} show that 
\begin{equation}\label{sPNGrR}
\prel 2 \subset \prel 3 \subset \prel 4\subset\dots \text.
\end{equation}
We obtain from Examples~\ref{krkrntNk} and \ref{cvmdfPnTs}, Theorem~\ref{thmmain}(A), and Proposition~\ref{prsoZHrW} that 
\begin{equation}\label{sdhGrR}
\kccc \subset \kccol\subset \kcplan\,\text{ and, clearly, }\prel 2\subseteq  \kcplan\text.
\end{equation}
\skc{If} $E\subseteq \mathbb R^2$ consist\skc{s} of three non-collinear points and their barycenter\skc{, then the convex geometry} $\pair E{\phullo 2E}$ does not satisfy \cprop 2, and we conclude from Proposition~\ref{prsoZHrW} that
\begin{equation}\label{dWh4PrR}
 \prel 2\not\subseteq \kccol\text.
\end{equation}
In the lattices associated with members of $\prel n$, all join-irreducible element\skc s are atoms. This implies that 
\begin{equation}\label{sdhhLvBr}
\text{for all } n\geq 2,\text{ }\,\,\kccc \not\subseteq \prel n\text.
\end{equation}

\subsection{Some open problems}
In spite of  \eqref{sPNGrR}, \eqref{sdhGrR}, \eqref{dWh4PrR},  and \eqref{sdhhLvBr}, we do not have a satisfactor\skc y description of the partially ordered set  
\begin{equation}\label{dkKFgM}
\pair{\set{ \kccc, \kccol, \kcplan, \kcall, \prel 1,\prel 2, \prel 3,\dots }}{\subseteq}\text.
\end{equation}
In particular, we do not know whether  \[\prel 3\overset{?}\subseteq \kcplan\quad \text{or} \quad  \kcplan \overset{?}= \kcall\quad \text{holds.}
\]
If we augment the set \eqref{dkKFgM} with convex geometries obtained from $n$-dimensional spheres or, say, coplanar three-dimensional spheres,  then the problem becomes even more difficult.
Finally, while Theorem~\ref{thmmain} describes $\kccc$ in an abstract way, we have no similar descriptions for  $\kccol$ and $\kcplan$. 
 
\subsection{\refegy{Representation by $1$-dimensional circles (added on May 19, 2013)}}
\label{kirasection}
\refegy{Let $F$ be a finite subset of $\set{\pair ab: a,b\in\mathbb R\text{ and }a\leq b}$. Its elements  will be called \emph{$1$-dimensional circles}. If $C=\pair ab\in F$, then $a=\lmpt C$ and $b=\rmpt C$ are the left and right endpoints of $C$, respectively.  Formulas \eqref{psxdFa} and \eqref{psxdFb} still make sense, and we clearly  obtain that $\pair F{\chullo F}$ is a convex geometry. When reading the first version of the present paper,  \init{K.\ }Adaricheva \cite{kirajan6} observed that Theorem~\ref{thmmain} has the following corollary.}

\begin{corollary}[\init{K.\ }\refegy{Adaricheva \cite{kirajan6}}]\label{kiracorol} \refegy{Up to isomorphism, finite convex geometries of convex dimension at most 2 are characterized as the convex geometries $\pair F{\chullo F}$, where $F\subseteq\set{\pair ab: a,b\in\mathbb R\text{ and }a\leq b}$ and $F$ is finite.}
\end{corollary}

\begin{proof} 
\refegy{Let $\pair U\Phi$ be a convex geometry of dimension at most 2.  Theorem~\ref{thmmain} yields a finite, separated, concave set $M$ of collinear circles such that $\pair U\Phi\cong\pair {M}{\chullo M}$. For $C\in M$, let $\phi(C)=\pair{\lmpt C}{\rmpt C}$. Let 
$F=\set{\phi(C): C\in M}$.
By \eqref{efbetween},  $\phi\colon \pair {M}{\chullo M}\to \pair {F}{\chullo F}$ is an isomorphism. This proves the non-trivial part; the trivial part has already been mentioned.}
\end{proof}

\refegy{As opposed to the proof above, it is far less easy to derive Theorem~\ref{thmmain} from Corollary~\ref{kiracorol}, because a finite set $F$ of 1-dimensional circles is rarely of the form $\set{\phi(C): C\in M}$ for a set $M$ of concave, collinear circles. 
}

\end{document}